\documentclass[a4paper, 10pt]{amsart}
\usepackage{amsmath}
\usepackage{amssymb, color}
\usepackage{hyperref}

\theoremstyle{plain}
\newtheorem{theorem}{\bf Theorem}[section]
\newtheorem{proposition}[theorem]{\bf Proposition}
\newtheorem{lemma}[theorem]{\bf Lemma}
\newtheorem{corollary}[theorem]{\bf Corollary}

\theoremstyle{definition}

\newcommand{\N}{\mathbb N}
\newcommand{\Z}{\mathbb Z}

\newcommand{\Q}{\mathbb Q}

\DeclareMathOperator{\lcm}{lcm}

\DeclareMathOperator{\Pic}{Pic}

\newcommand{\red}{{\text{\rm red}}}

\numberwithin{equation}{section}

\begin{document}

\title[A realization theorem for sets of lengths]{A realization theorem for sets of lengths \\ in numerical monoids}

\address{Institute for Mathematics and Scientific Computing\\ University of Graz, NAWI Graz\\ Heinrichstra{\ss}e 36\\ 8010 Graz, Austria }
\email{alfred.geroldinger@uni-graz.at}
\urladdr{http://imsc.uni-graz.at/geroldinger}

\address{Universit{\'e} Paris 13 \\ Sorbonne Paris Cit{\'e} \\ LAGA, CNRS, UMR 7539,  Universit{\'e} Paris 8\\ F-93430, Villetaneuse, France \\ and \\ Laboratoire Analyse, G{\'e}om{\'e}trie et Applications (LAGA, UMR 7539) \\ COMUE  Universit{\'e} Paris Lumi{\`e}res \\  Universit{\'e} Paris 8, CNRS \\  93526 Saint-Denis cedex, France} \email{schmid@math.univ-paris13.fr}

\author{Alfred Geroldinger  and Wolfgang A. Schmid}

\thanks{This work was supported by the Austrian Science Fund FWF, Project Number P 28864-N35}



\begin{abstract}
We show that for every finite nonempty subset $L$ of $\N_{\ge 2}$ there are a numerical monoid $H$ and a squarefree element $a \in H$ whose set of lengths $\mathsf L (a)$ is equal to $L$.
\end{abstract}

\maketitle

\bigskip
\section{Introduction} \label{1}
\bigskip

In the last decade the arithmetic of numerical monoids has found wide interest in the literature.  Since numerical monoids are finitely generated,  every element of a given monoid can be written as a sum of atoms and all arithmetical invariants describing the non-uniqueness of factorizations are finite. The focus of research was on obtaining precise values for the arithmetical invariants (e.g., \cite{As-GS16, Om12a, Ba-Ne-Pe17a}), on their interplay with minimal relations of a given presentation (e.g., \cite{C-G-L09}), and also on computational aspects (e.g., \cite{GG-MF-VT15, GS16a}  and \cite{numericalsgps} for a software package in GAP). A further direction of research was to establish realization results for arithmetical parameters. This means to show that there are numerical monoids whose arithmetical parameters have prescribed values. So for example, it was proved only recently that every finite set (with some obvious restrictions) can be realized as the set of catenary degrees of a numerical monoid (\cite{ON-Pe18a}). The goal of the present note is to show a realization theorem for sets of lengths.

Let $H$ be a numerical monoid. If $a \in H$ and $a=u_1+\ldots+u_k$, where $u_1, \ldots, u_k$ are atoms of $H$, then $k$ is called a factorization length of $a$ and the set $\mathsf L (a) \subset \N$ of all factorization lengths is called the set of lengths of $a$. Further, $\mathcal L (H) = \{\mathsf L (a) \mid a \in H \}$ denotes the system of sets of lengths of $H$. It is easy to see that all sets of lengths are finite nonempty and  can get arbitrarily large, and it is well-known that they have a well-defined structure (see the beginning of Section \ref{3}). As a converse, we show in the present paper that for every finite nonempty set $L \subset \N_{\ge 2}$ there is a numerical monoid $H$ and a squarefree element $a \in H$ such that $\mathsf L (a) = L$ (Theorem \ref{3.3}). In fact, we show more precisely that the number of factorizations of each length can be prescribed.  Several types of realization results for sets of lengths are known in the literature, most of them in the setting of  Krull monoids (see \cite{Ka99a,  Ge-Ha-Le07, Sc09a, Fr13a, Fr-Na-Ri18a}, \cite[Theorem 7.4.1]{Ge-HK06a}). However, we know that if $H$ is a numerical monoid, then $\mathcal L (H) \ne \mathcal L (H')$ for every  Krull monoid $H'$ (see \cite[Theorem 5.5]{Ge-Sc-Zh17b} and note that every numerical monoid is strongly primary).

It is an open problem which finite sets of positive integers can occur as sets of distances of numerical monoids. Based on our main result we can show that every finite set is contained in a set of distances of a numerical monoid (Corollary \ref{3.4}). There is a vibrant interplay between numerical monoids, and more generally affine monoids, and the associated semigroup algebras (\cite{Ba-Do-Fo97, Ba06b, Br-Gu09a}). In Corollary \ref{3.5} we shift our realization result from numerical monoids to numerical semigroup algebras.

\medskip
\section{Background on the arithmetic of numerical monoids} \label{2}
\medskip

We denote by $\mathbb P \subset \N \subset \Z \subset \Q$ the set of prime numbers, positive integers, integers, and rational numbers respectively. For $a, b \in \Q$, let $[a, b] = \{ x \in \Z \mid a \le x \le b \}$ be the discrete interval of  integers lying between $a$ and $b$. If $A, B \subset \Z$, then $A+B=\{a+b \mid a \in A, b \in B\}$ denotes the sumset and $kA = A+ \ldots + A$ is the $k$-fold sumset for every $k \in \N$.
If $A = \{m_1, \ldots, m_k \} \subset \Z$ with $m_{i-1} < m_{i}$ for each $i \in [2,k]$, then $\Delta (A) = \{m_i - m_{i-1} \mid i \in [2,k] \} \subset \N$ is the set of distances of $L$. Note that $\Delta (A)=\emptyset$ if and only if $|A| \le 1$.

By a {\it monoid}, we mean a commutative cancellative semigroup with identity element. Let $H$ be a monoid. Then $H^{\times}$ denotes the group of invertible elements, $\mathsf q (H)$ the quotient group of $H$, and $\mathcal A (H)$ the set of atoms (irreducible elements) of $H$. We say that $H$ is reduced if the identity element is the only invertible element. We call $H_{\red} =H/H^{\times}$ the reduced monoid associated to $H$. A {\it numerical monoid} is a submonoid of $(\N_0, +)$ whose complement in $\N_0$  is finite. Every numerical monoid is finitely generated, reduced, and its quotient group is $\Z$. For any set $P$, let $\mathcal F (P)$ denote the free abelian monoid with basis $P$. Then, using additive notation, every element $a \in \mathsf q ( \mathcal F (P) )$ can be written uniquely in the form
\[
a = \sum_{p \in P} l_p  p \,,
\]
where $l_p \in \Z$ for each $p \in P$, and all but finitely many $l_p$ are equal to $0$. For $a = \sum_{p \in P} l_p  p \in \mathcal F (P)$,  we set $|a|= \sum_{p\in P} l_p \in \N_0$ and call it the length of $a$.

We recall some arithmetical concept of monoids. Since our focus is on numerical monoids we use additive notation. Let $H$ be an additively written  monoid. The (additively written) free abelian monoid $\mathsf Z (H)= \mathcal F ( \mathcal A (H_{\red}))$ is called the {\it factorization monoid} of $H$ and the canonical epimorphism $\pi \colon \mathsf Z (H) \to H_{\red}$ is the factorization homomorphism. For $a
\in H$ and $k\in\N$,
\[
\begin{aligned}
\mathsf Z_H (a) = \mathsf Z (a) & =\pi^{-1} (a+H^\times)\subset \mathsf Z (H)\quad \text{is the\ {\it set of factorizations}\ of\ $a$}\,,\\
\mathsf Z_{H,k} (a)= \mathsf Z_k (a) & =\{ z\in\mathsf Z (a)\mid |z| = k\}\quad \text{is the\ {\it set of factorizations}\ of\ $a$ of length\ $k$},\quad\text{and} \\
\mathsf L_H (a) = \mathsf L (a) & =\bigl\{ |z|\,\bigm|\, z\in \mathsf Z (a)\bigr\}\subset\N_0\quad\text{is the\ {\it set of  lengths}\ of $a$}\,.
\end{aligned}
\]
Thus, by definition, $\mathsf L (a) = \{0\}$ if and only if $a \in H^{\times}$  and $\mathsf L (a)=\{1\}$ if and only if $a \in \mathcal A (H)$.
The monoid $H$ is said to be atomic if $\mathsf Z (a) \ne \emptyset$ for all $a \in H$ (equivalently, every non invertible element is a finite sum of atoms).
We call
\begin{itemize}
\item $\mathcal L (H) = \{ \mathsf L (a) \mid a \in H\}$ the {\it system of sets of lengths} of $H$, and

\item $\Delta (H) = \bigcup_{L \in \mathcal L (H)} \Delta (L)$ the {\it set of distances} (also called {\it delta set})  of $H$.
\end{itemize}
Every numerical monoid $H$ is atomic with finite set of distances $\Delta (H)$, and $\Delta (H)=\emptyset$ if and only if   $H = \N_0$.

\medskip
\section{A realization theorem for sets of lengths} \label{3}
\medskip

The goal of this section is to prove our main realization theorem, namely that for every finite nonempty subset $L \subset \N_{\ge 2}$ there exists a numerical monoid $H$ such that $L$ is a set of lengths of $H$ (Theorem \ref{3.3}). We show the existence of this monoid by an explicit recursive construction over the size of $L$. Instead of working with numerical monoids directly, we work in the setting of finitely generated additive submonoids of the nonnegative rationals. Additive submonoids of $(\Q_{\ge 0}, +)$ are called Puiseux monoids and have recently been studied in a series of papers by F. Gotti et al. (e.g., \cite{Go17a}). In the setting of Puiseux monoids all arithmetical concepts refer to addition and not to multiplication of rationals. In particular,
an element $a$ of a Puiseux monoid $H$ is said to be squarefree if there are no nonzero elements $b, c \in H$ such that $a = b+b+c$.

Clearly, the constructed numerical monoid  heavily depends on the given set $L$. This is inevitable because for every fixed numerical monoid $H$, sets of lengths have a well-defined structure.  Indeed, there is a constant $M \in \N_0$ (just depending on $H$) such that every $L \in \mathcal L (H)$ has the form
\begin{equation} \label{structure}
L = y + \bigl( L' \cup \{\nu d \mid \nu  \in [0,l] \} \cup L''
\bigr) \ \subset \ y + d \Z \,,
\end{equation}
where $d = \min \Delta (H)$, $y \in \mathbb Z$, $L' \subset [-M, -1]$, and $L'' \subset l d + [1, M]$ (\cite[Theorem 4.3.6]{Ge-HK06a}).

\smallskip
We start with a technical lemma.

\medskip
\begin{lemma} \label{3.1}
Let $k \in \N_{\ge 2}$. Then there exist pairwise distinct nonzero $c_1, \ldots, c_k \in [-k^{k-1}, k^{k-1}]$ with $c_1+ \ldots + c_k=0$ such that for all primes $p > (k+1)k^{k-1}$ the following property holds{\rm \,:} if $l_1, \ldots, l_k \in \N_0$ such that $\sum_{i=1}^k l_ic_i \equiv 0 \mod p$, then
\begin{itemize}
\item[] $l_1= \ldots =  l_k=0$ \ or \ $l_1= \ldots =  l_k=1$ \ or \  $l_1+ \ldots + l_k > k$.
\end{itemize}
\end{lemma}

\begin{proof}
For $i \in [1,k-1]$ we define $c_i = k^{i-1}$, and we set $c_k = - \sum_{i=1}^{k-1} c_i$. Then clearly,
\[
c_k = - \sum_{i=1}^{k-1} c_i = - \sum_{i=1}^{k-1} k^{i-1} = - \frac{k^{k-1}-1}{k-1} \,.
\]
Now we choose a prime $p > (k+1)k^{k-1}$ and $l_1, \ldots, l_k \in \N_0$ such that $\sum_{i=1}^k l_ic_i \equiv 0 \mod p$
and  $\sum_{i=1}^k l_i > 0$. We may distinguish the following two cases.

\noindent
CASE 1: \ $\sum_{i=1}^{k-1}l_ic_i \ge p$ or $l_kc_k \le -p$.

If $p \le \sum_{i=1}^{k-1}l_ic_i \le \big( \sum_{i=1}^{k-1}l_i \big) c_{k-1}$, then
\[
\sum_{i=1}^{k}l_i \ge \sum_{i=1}^{k-1}l_i \ge \frac{p}{c_{k-1}} > \frac{(k+1)k^{k-1}}{c_{k-1}} \ge k+1 \,.
\]
If $p \le l_k|c_k|$, then
\[
\sum_{i=1}^{k}l_i \ge l_k \ge \frac{p}{|c_k|} > \frac{(k+1)k^{k-1}}{|c_{k}|} \ge k+1 \,.
\]

\smallskip
\noindent
CASE 2: \ $\sum_{i=1}^{k-1}l_ic_i < p$ and $l_kc_k > -p$.

Since $\sum_{i=1}^k l_i c_i \equiv 0 \mod p$, we infer that $\sum_{i=1}^k l_i c_i = 0$. Suppose that there is a $j \in [1,k]$ with $l_j \ge k$. Since at least two elements of $l_1, \ldots, l_k$ are positive, it follows that   $ \sum_{i=1}^k l_{i} > k$. Suppose that $l_i \in [0, k-1]$ for all $i \in [1, k]$. Since $0 = \sum_{i=1}^k l_ic_i$, the definition of $c_1, \ldots, c_k$ implies that
\[
\sum_{i=1}^{k-1} l_i k^{i-1} = \sum_{i=1}^{k-1}l_k k^{i-1} \,.
\]
By the uniqueness of the $k$-adic digit expansion, we infer  that $l_i=l_k$ for all $i \in [1,k-1]$. If $l_1=1$, then $l_1= \ldots = l_k=1$. If $l_1> 1$, then $l_1+\ldots +l_k=kl_1 > k$.
\end{proof}

The following proposition will be our key tool to do the recursive construction step in Theorem \ref{3.3}.
For every prime $p \in \mathbb P$, we denote by $\mathsf v_p$ the usual $p$-adic valuation of the rationals, that is, for  $q \in \mathbb{Q}\setminus \{0\}$,  $\mathsf v_p(q)$ the integer $j$ such that $q = p^j \frac{a}{b}$ with integers $a,b$ such that $p \nmid ab$. Moreover,  we set $\mathsf v_p(0)= \infty$.

\medskip
\begin{proposition} \label{3.2}
Let $k \in \N_{\ge 2}$ and $H \subset (\Q_{\ge 0}, +)$ be a finitely generated monoid with $\N_0 \subset H$ and $\mathcal A (H) \subset \Q_{<1}$. Then there exists a finitely generated monoid $H'$ with $H \subset H'$ and $\mathcal A (H) \subset \mathcal A (H') \subset \Q_{<1}$ such that the following properties are satisfied{\rm \,:}
\begin{enumerate}
\item[(a)] For all $u \in H$ with $u < 1$ we have $\mathsf Z_H (u) = \mathsf Z_{H'} (u)$.

\smallskip
\item[(b)] $\mathsf Z_{H'}(1) = \mathsf Z_H (1) \uplus \{q_1 + \ldots + q_k\}$, where $q_1, \ldots, q_k$ are pairwise distinct and  $\mathcal A (H') = \mathcal A (H) \uplus \{q_1, \ldots, q_k\}$.
\end{enumerate}
\end{proposition}

\begin{proof}
We set
\[
\mathcal A (H) = \Big\{ \frac{a_1}{b_1}, \ldots, \frac{a_s}{b_s} \Big\}
\]
where $a_i, b_i \in \N$ with $\gcd (a_i,b_i)=1$ for all $i \in [1,s]$.  Let $c_1, \ldots, c_k \in [-k^{k-1}, k^{k-1}]$  such that all properties of Lemma \ref{3.1} are satisfied.
We choose a prime number $p \in \N$ such that
\[
p \nmid \lcm (b_1, \ldots, b_s) \quad \text{and} \quad p  > (k+1)k^{k-1} \,,
\]
and we define
\[
q_i = \frac{p+c_i}{kp} \quad \text{for every} \quad i \in [1,k] \,.
\]
By construction, we have $q_1+ \ldots + q_k=1$ and $\mathsf v_p (q_i) = -1$ whence $q_i \notin H$ for all $i \in [1,k]$. We define
\[
H' = \big[ H, q_1, \ldots, q_k \big] \subset (\Q_{\ge 0}, +)
\]
to be the additive submonoid of nonnegative rationals generated by the elements of $H$ and by $q_1, \ldots, q_k$. Thus $H'$ is generated by $\mathcal A (H) \cup \{q_1, \ldots, q_k\}$ whence finitely generated. Since $H'$ is reduced, \cite[Proposition 1.1.7]{Ge-HK06a} implies that $H'$ is atomic and
\begin{equation} \label{main-inclusion}
\mathcal A (H') \subset \mathcal A (H) \cup \{q_1, \ldots, q_k\} \,.
\end{equation}
We continue with the following assertions.

\smallskip
\begin{enumerate}
\item[{\bf A1.}\,] $\{q_1, \ldots, q_k\} \subset \mathcal A (H')$.

\item[{\bf A2.}\,] Let $u \in H$ and suppose that $u$ has a factorization $z \in \mathsf Z_{H'} (u)$ which is divisible by some element from $\{q_1, \ldots, q_k\}$. Then either $u>1$ or $z=q_1+ \ldots + q_k \in \mathsf Z_{H'} (u)$ (whence in particular $u=1$).
\end{enumerate}

\noindent
{\it Proof of} \,{\bf A1}.\, Assume to the contrary that there is an $i \in [1,k]$ such that $q_i \notin \mathcal A (H')$. Since $q_i \notin H$, it is divisible by an atom from $\mathcal A (H') \setminus \mathcal A (H) \subset \{q_1, \ldots, q_k\}$, say  $q_i = q_j + b$ with $j \in [1,k] \setminus \{i\}$ and  $b \in H' \setminus \{0\}$. We claim that $b \notin H$. Since $0 \ne b = q_i - q_j$ and $0 \ne |c_i-c_j| \le 2 k^{k-1} < p$,
\[
\mathsf v_p (q_i-q_j) = \mathsf v_p \Big( \frac{c_i-c_j}{kp} \Big) = -1
\]
which implies that $b \notin H$. Thus there is an $l \in [1,k]$ such that $b = q_l + d$ with $d \in H' \subset \Q_{\ge 0}$. Since $p > (k+1)k^{k-1} \ge 3k^{k-1} \ge |c_i|+|c_j|+|c_l|$, it follows that
\[
q_i   = q_j + q_l + d \ge q_j+q_l = \frac{2p+c_j+c_l}{kp} > \frac{p+c_i}{kp} = q_i \,,
\]
a contradiction. \qed[Proof of {\bf A1}]

\smallskip
\noindent
{\it Proof of} \,{\bf A2}.\, Since $u$ has a factorization which is divisible by some element from $\{q_1, \ldots, q_k\}$, there are $l_1, \ldots, l_k \in \N_0$ and $v \in H$ such that
\[
u = v + \sum_{i=1}^k l_i q_i \quad \text{and} \quad \sum_{i=1}^k l_i > 0 \,.
\]
Since $\mathsf v_p (u) \ge 0$ and $\mathsf v_p (v) \ge 0$, it follows that
\[
0 \le \mathsf v_p (u-v)= \mathsf v_p \Big( \sum_{i=1}^k l_i q_i \Big) = \mathsf v_p \Big( \frac{ \sum_{i=1}^k l_ip + \sum_{i=1}^k l_ic_i}{kp} \Big)
\]
whence $\sum_{i=1}^k l_ic_i \equiv 0 \mod p$. Therefore Lemma \ref{3.1} implies that
\[
l_1= \ldots = l_k =1 \quad \text{or} \quad  \sum_{i=1}^k l_i > k \,.
\]
If $\sum_{i=1}^k l_i > k$ and $j \in [1,k]$ with $q_j = \min \{q_1, \ldots, q_k\}$, then
\begin{equation} 
u = v + \sum_{i=1}^l l_iq_i \ge (k+1) q_j = (k+1) \frac{p+c_j}{kp} > 1 \,,
\end{equation}
where the last inequality uses that $p > (k+1)k^{k-1} \ge (k+1)|c_j|$.  If $l_1= \ldots = l_k =1$, then
\[
u = \sum_{i=1}^l l_iq_i + v = q_1 + \ldots + q_k + v = 1 + v \,.
\]
Thus $v>0$ implies $u>1$ and $v=0$ implies $u=1$ and $z=q_1+ \ldots + q_k$.
\qed[Proof of {\bf A2}]

\smallskip
If $u \in \mathcal A (H)$, then $u < 1$ by assumption and {\bf A2} implies  that $u$ is not divisible by any element from $\{q_1, \ldots, q_k\}$ and therefore $u \in \mathcal A (H')$.  Thus we obtain that     $\mathcal A (H) \subset \mathcal A (H')$ and together with {\bf A1} and \eqref{main-inclusion}, it follows  that
\begin{equation} \label{main2}
\mathcal A (H') = \mathcal A (H) \uplus \{q_1, \ldots, q_k\} \,.
\end{equation}
Thus,  we have that
\begin{equation} \label{main3}
\mathsf Z (H) = \mathcal F (\mathcal A (H)) \subset \mathcal F ( \mathcal A (H')) = \mathsf Z (H') \quad \text{and} \quad  \mathsf Z_H (u) \subset \mathsf Z_{H'} (u)
\end{equation}
for every $u \in H$.
If  $u < 1$, then {\bf A2} implies that $\mathsf Z_H (u) = \mathsf Z_{H'} (u)$.

It remains to show Property (b) given in Proposition \ref{3.2}, namely that
\[
\mathsf Z_H (1) \uplus \{q_1 + \ldots + q_k\} = \mathsf Z_{H'}(1) \,.
\]
We see from Equation \eqref{main3} that $\mathsf Z_H (1) \uplus \{q_1 + \ldots + q_k\} \subset \mathsf Z_{H'}(1)$. Conversely, let $z$ be a factorization of $1$ in $H'$. Then either $z \in \mathsf Z_H (1)$ or $z$ is divisible (in $\mathsf Z (H')$) by some element from $\{q_1, \ldots, q_k\}$. In the latter case {\bf A2} implies that $z=q_1+ \ldots + q_k \in \mathsf Z (H')$.
\end{proof}

\medskip
\begin{theorem} \label{3.3}
Let $L \subset \N_{\ge 2}$ be a finite nonempty set and $f \colon L \to \N$ a map. Then there exist a numerical monoid $H$ and a squarefree element $a \in H$ such that
\begin{equation} \label{mainproperty}
\mathsf L (a) = L \quad \text{and} \quad |\mathsf Z_k (a)|= f (k) \quad \text{for every} \ k \in L \,.
\end{equation}
\end{theorem}

\begin{proof}
Every finitely generated submonoid of $(\Q_{\ge 0}, +)$ is isomorphic to a numerical monoid (cf. \cite[Proposition 3.2]{Go17a}) and the isomorphism maps squarefree elements onto squarefree elements. Thus it is sufficient to show that, for every set $L$ and every map $f$ as in the statement of the theorem,  there is  a finitely generated submonoid $H$ of the nonnegative rationals with $\N_0 \subset H$ and $\mathcal A (H) \subset \Q_{< 1}$ such that the element $a=1 \in H$ is squarefree in $H$ and has the properties given in \eqref{mainproperty}.

Clearly, it is equivalent to consider nonzero maps  $f \colon \N_{\ge 2} \to \N_0$ with finite support and to find  a  monoid $H$ as above such that $|\mathsf Z_k (1)|= f (k)$ for every $k \in \N_{\ge 2}$ and $1$ is squarefree in $H$.
For every nonzero  map $f \colon \N_{\ge 2} \to \N_0$ with finite support   $\sum_{k \ge 2} f (k)$ is a positive integer
and we proceed by induction on this sum.

To do the base case, let $f \colon \N_{\ge 2} \to \N_0$ be a map with $\sum_{k \ge 2} f (k) = 1$.
Let $k \in \N_{\ge 2}$ with $f (k)=1$. We have to find a finitely generated monoid $H \subset (\Q_{\ge 0}, +)$ with $\mathcal A (H) \subset \Q_{<1}$ and pairwise distinct atoms $q_1, \ldots, q_k \in H$ such that  $\mathsf Z_H (1) = \{q_1+ \ldots + q_k\}$.

We proceed along the lines of the proof  of {\bf A2} in Proposition \ref{3.2}. Indeed,
we choose $c_1, \ldots, c_k \in [-k^{k-1}, k^{k-1}]$  such that all properties of Lemma \ref{3.1} are satisfied and pick a prime number $p \in \N$ with $ p  > (k+1)k^{k-1}$. We set
\[
q_i = \frac{p+c_i}{kp} \quad \text{for every} \quad i \in [1,k]
\]
and define $H = [q_1, \ldots, q_k] \subset \Q_{\ge 0}$.  By \cite[Proposition 1.1.7]{Ge-HK06a}, $\mathcal A (H) \subset \{q_1, \ldots, q_k\}$.  Since for all (not necessarily distinct) $r,s,t \in [1,k]$ we have $q_r < q_s+q_t$, it follows that $q_r \in \mathcal A (H)$. Thus we obtain that  $\mathcal A (H) = \{q_1, \ldots, q_k\}$. Since $q_1+ \ldots + q_k = 1$, it follows that $\{q_1+ \ldots + q_k\} \subset \mathsf Z_H (1)$. To show equality, let $l_1, \ldots, l_k \in \N_0$ such that $1 = \sum_{i=1}^k l_iq_i$.
It follows that $\sum_{i=1}^k l_ic_i \equiv 0 \mod p$. Therefore Lemma \ref{3.1} implies that
\[
l_1= \ldots = l_k =1 \quad \text{or} \quad  \sum_{i=1}^k l_i > k \,.
\]
If $\sum_{i=1}^k l_i > k$ and $j \in [1,k]$ with $q_j = \min \{q_1, \ldots, q_k\}$, then
\begin{equation} 
1 =  \sum_{i=1}^l l_iq_i \ge (k+1) q_j = (k+1) \frac{p+c_j}{kp} > 1 \,,
\end{equation}
a contradiction. Thus $l_1= \ldots = l_k =1$ and the claim follows.

Now let $N \in \N_{\ge 2}$ and suppose that the assertion holds  all nonzero maps $f \colon \N_{\ge 2} \to \N_0$ with finite support and with  $\sum_{k \ge 2} f (k) < N$. Let $f_0 \colon \N_{\ge 2} \to \N_0$ with $\sum_{k \ge 2} f_0 (k) = N$. We choose an element $k_0 \in \N_{\ge 2}$ with $f (k_0) \ne 0$ and define a map $f_1 \colon \N_{\ge 2} \to \N_0$ as $f_1 (k_0) = f_0 (k_0)-1$ and $f_1 (k) = f_0 (k)$ for all $k \in \N_{\ge 2} \setminus \{k_0\}$. By the induction hypothesis, there exists a finitely generated monoid $H_1 \subset (\Q_{\ge 0}, +)$ with $\N_0 \subset H_1$ and $\mathcal A (H_1) \subset \Q_{<1}$  such that $|\mathsf Z_{H_1, k} (1)|= f_1 (k)$ for every $k \in \N_{\ge 2}$ and $1$ is squarefree in $H_1$. By Proposition \ref{3.2} there exist a finitely generated monoid $H_0 \subset (\Q_{\ge 0}, +)$ such that
\[
\mathsf Z_{H_0}(1) = \mathsf Z_{H_1} (1) \uplus \{q_1 + \ldots + q_{k_0} \},
\]
where $q_1, \ldots, q_{k_0}$ are pairwise distinct and $\mathcal A (H_0) = \mathcal A (H_1) \uplus \{q_1, \ldots, q_{k_0} \} \subset \Q_{< 1}$.
Since $q_1, \ldots, q_{k_0}$ are pairwise distinct and since $1$ was squarefree in $H_1$, it follows that $1$ is squarefree in $H_0$. Moreover, $\mathsf Z_{H_0, k}(1) = \mathsf Z_{H_1, k} (1)$ for all $k \in \N_{\ge 2} \setminus \{k_0\}$ and $\mathsf Z_{H_0, k_0}(1) = \mathsf Z_{H_1, k_0} (1) \uplus \{q_1 + \ldots + q_{k_0} \}$. In particular, we have
$|\mathsf Z_{H_0, k} (1)|= f_0 (k)$ for every $k \in \N_{\ge 2}$.
\end{proof}

\smallskip
We continue with a corollary on sets of distances. Let $H$ be an atomic monoid with nonempty set of distances $\Delta (H)$. Then it is easy to verify that $\min \Delta (H) = \gcd \Delta (H)$, and the question is which finite sets $D$ with $\min D = \gcd D$ can be realized as a set of distances in a given class of monoids or domains. The question has an affirmative answer in the class of finitely generated Krull monoids (\cite{Ge-Sc17a}).  If $H$ is a numerical monoid generated by two atoms, say $\mathcal A (H)= \{n_1, n_2\}$, then $\Delta (H) = \{|n_2-n_1|\}$ whence every singleton occurs as a set of distances of a numerical monoid. There are periodicity results on individual sets $\Delta (\mathsf L (a))$ for elements in a numerical monoid (\cite{Sc-Ho-Ka09a}), but the only realization result beyond the simple observation above is due to Colton and Kaplan (\cite{Co-Ka17a}).  They show that every two-element set $D$ with $\min D = \gcd D$ can be realized as the set of distances of a numerical monoid. As a consequence of Theorem \ref{3.3} we obtain that every finite set is contained in the set of distances of a numerical monoid (this was achieved first by explicit constructions in \cite[Corollary 4.8]{B-C-K-R06}).

\medskip
\begin{corollary} \label{3.4}
For every finite nonempty subset $D \subset \N$ there is a numerical monoid $H$ such that $D \subset \Delta (H)$.
\end{corollary}

\begin{proof}
Let $D = \{d_1, \ldots, d_k\} \subset \N$ be a finite nonempty subset. By Theorem \ref{3.3} there is a numerical monoid $H$  such that $L = \{ 2, 2+d_1, 2+d_1+d_2, \ldots, 2+d_1+ \ldots + d_k \} \in \mathcal L (H)$ whence $D = \Delta ( L ) \subset \Delta (H)$.
\end{proof}

\smallskip
Let $K$ be a field and $H$ a numerical monoid. The semigroup algebra
\[
K[H] = \Big\{ \sum_{h \in H} a_h X^h \mid a_h \in H \ \text{for all} \ h \in H \ \text{and almost all $a_h$ are zero} \Big\} \subset K[X]
\]
is a one-dimensional noetherian domain and its integral closures $K[X]$ is a finitely generated module over $K[H]$. Thus $K[H]$ is weakly Krull, $\Pic ( K[H] )$ is finite if $K$ is finite whence it  satisfies all arithmetical finiteness results established for weakly Krull Mori domains with finite class group (see \cite{Ge-HK06a} for basic information). However, all results on $\mathcal L \big( K[H] \big)$ so far depend on detailed information on the Picard group and the distribution of height one prime ideals not containing the conductor in the Picard group.

\medskip
\begin{corollary} \label{3.5}
Let $K$ be a field, $L \subset \N_{\ge 2}$  a finite nonempty set,  and $f \colon L \to \N$ a map. Then there is a numerical monoid $H$ and a squarefree element $g \in K[H]$ such that
\[
\mathsf L_{K[H]} (g)  = L \quad \text{and} \quad |\mathsf Z_{K[H],k} (g)|= f (k) \quad \text{for every} \ k \in L \,.
\]
\end{corollary}

\begin{proof}
By Theorem \ref{3.3} there is a numerical monoid $H$ and a squarefree element $c \in H$ having the required properties. Clearly, the additive monoid $H$ is isomorphic to the multiplicative monoid of monomials
\[
H'= \{X^h \mid h \in H \} \subset K[H]  \,.
\]
Since $K[H]^{\times} = K^{\times}$, the monoid $H''= \{ cX^h \mid h \in H, c \in K^{\times} \} \subset K[H]$ is a divisor-closed submonoid and  $H''_{\red} \cong H'$. Thus for every $k \in L$ and the element $g=X^c \in K[H]$ we obtain that
\[
|\mathsf Z_{K[H],k}(X^c)| = |\mathsf Z_{H'',k}(X^c)|=|\mathsf Z_{H',k}(X^c)|=|\mathsf Z_{H,k}(c)|=f (k)
\]
whence the assertion follows.
\end{proof}

\providecommand{\bysame}{\leavevmode\hbox to3em{\hrulefill}\thinspace}
\providecommand{\MR}{\relax\ifhmode\unskip\space\fi MR }
\providecommand{\MRhref}[2]{%
  \href{http://www.ams.org/mathscinet-getitem?mr=#1}{#2}
}
\providecommand{\href}[2]{#2}


\begin{thebibliography}{10}

\bibitem{As-GS16}
A.~Assi and P.~A. Garc{\'i}a-S\'anchez, \emph{Numerical semigroups and
  applications}, RSME Springer Series, vol.~1, Springer, [Cham], 2016.

\bibitem{Ba-Ne-Pe17a}
T.~Barron, C.O'Neill, and R.~Pelayo, \emph{On the set of elasticities in
  numerical monoids}, Semigroup Forum \textbf{94} (2017), 37 –-- 50.

\bibitem{Ba06b}
V.~Barucci, \emph{Numerical semigroup algebras}, Multiplicative {I}deal
  {T}heory in {C}ommutative {A}lgebra (J.W. Brewer, S.~Glaz, W.~Heinzer, and
  B.~Olberding, eds.), Springer, 2006, pp.~39 -- 53.

\bibitem{Ba-Do-Fo97}
V.~Barucci, D.E. Dobbs, and M.~Fontana, \emph{Maximality {P}roperties in
  {N}umerical {S}emigroups and {A}pplications to {O}ne-{D}imensional
  {A}nalytically {I}rreducible {L}ocal {D}omains}, vol. 125, Memoirs of the
  Amer. Math. Soc., 1997.

\bibitem{B-C-K-R06}
C.~Bowles, S.T. Chapman, N.~Kaplan, and D.~Reiser, \emph{On delta sets of
  numerical monoids}, J. Algebra Appl. \textbf{5} (2006), 695 -- 718.

\bibitem{Br-Gu09a}
W.~Bruns and J.~Gubeladze, \emph{Polytopes, {R}ings, and {K-T}heory}, Springer,
  2009.

\bibitem{C-G-L09}
S.T. Chapman, P.A. Garc{\'i}a-S{\'a}nchez, and D.~Llena, \emph{The catenary and
  tame degree of numerical monoids}, Forum Math. \textbf{21} (2009), 117 --
  129.

\bibitem{Sc-Ho-Ka09a}
S.T. Chapman, R.~Hoyer, and N.~Kaplan, \emph{Delta sets of numerical monoids
  are eventually periodic}, Aequationes Math. \textbf{77} (2009), 273 -- 279.

\bibitem{Co-Ka17a}
S.~Colton and N.~Kaplan, \emph{The realization problem for delta sets of
  numerical monoids}, J. Commut. Algebra \textbf{9} (2017), 313 -- 339.

\bibitem{numericalsgps}
M.~Delgado, P.A. Garc{\'i}a-S{\'a}nchez, and J.~Morais,
  \emph{``numericalsgps'': a {\sf {g}{a}{p}} package on numerical semigroups},
  \verb+(http://www.gap-system.org/Packages/numericalsgps.html)+.

\bibitem{Fr13a}
S.~Frisch, \emph{A construction of integer-valued polynomials with prescribed
  sets of lengths of factorizations}, Monatsh. Math. \textbf{171} (2013), 341
  -- 350.

\bibitem{Fr-Na-Ri18a}
S.~Frisch, S.~Nakato, and R.~Rissner,
\emph{Integer-valued polynomials on rings of algebraic integers of number fields with prescribed sets of lengths of factorizations},
\verb+(https://arxiv.org/abs/1710.06783)+.



\bibitem{GG-MF-VT15}
J.~I. Garc{\'i}a-Garc{\'i}a, M.~A. Moreno-Fr{\'i}as, and A.~Vigneron-Tenorio,
  \emph{Computation of delta sets of numerical monoids}, Monatsh. Math.
  \textbf{178} (2015), 457--472.

\bibitem{GS16a}
P.A. Garc{\'i}a-S{\'a}nchez, \emph{An overview of the computational aspects of
  nonunique factorization invariants}, Multiplicative {I}deal {T}heory and
  {F}actorization {T}heory (S.T. Chapman, M.~Fontana, A.~Geroldinger, and
  B.~Olberding, eds.), Springer, 2016, pp.~159 -- 181.

\bibitem{Ge-HK06a}
A.~Geroldinger and F.~Halter-Koch, \emph{Non-{U}nique {F}actorizations.
  {A}lgebraic, {C}ombinatorial and {A}nalytic {T}heory}, Pure and Applied
  Mathematics, vol. 278, Chapman \& Hall/CRC, 2006.

\bibitem{Ge-Ha-Le07}
A.~Geroldinger, W.~Hassler, and G.~Lettl, \emph{On the arithmetic of strongly
  primary monoids}, Semigroup Forum \textbf{75} (2007), 567 -- 587.

\bibitem{Ge-Sc17a}
A.~Geroldinger and W.~A. Schmid, \emph{A realization theorem for sets of
  distances}, J. Algebra \textbf{481} (2017), 188 -- 198.

\bibitem{Ge-Sc-Zh17b}
A.~Geroldinger, W.~A. Schmid, and Q.~Zhong, \emph{Systems of sets of lengths:
  transfer {K}rull monoids versus weakly {K}rull monoids}, In: Fontana M.,
  Frisch S., Glaz S., Tartarone F., Zanardo P. (eds) Rings, Polynomials, and
  Modules, Springer, Cham, 2017, pp.~191 -- 235.

\bibitem{Go17a}
F.~Gotti, \emph{On the atomic structure of {P}uiseux monoids}, Journal of
  Algebra and its applications \textbf{16} (2017), No. 07, 1750126.

\bibitem{Ka99a}
F.~Kainrath, \emph{Factorization in {K}rull monoids with infinite class group},
  Colloq. Math. \textbf{80} (1999), 23 -- 30.

\bibitem{Om12a}
M.~Omidali, \emph{The catenary and tame degree of numerical monoids generated
  by generalized arithmetic sequences}, Forum Math. \textbf{24} (2012), 627 --
  640.

\bibitem{ON-Pe18a}
C.~O'Neill and R.~Pelayo, \emph{Realizable sets of catenary degrees of
  numerical monoids}, Bull. Australian Math. Soc., to appear.

\bibitem{Sc09a}
W.A. Schmid, \emph{A realization theorem for sets of lengths}, J. Number Theory
  \textbf{129} (2009), 990 -- 999.

\end{thebibliography}
\end{document}